\documentclass[12pt]{amsart}
\usepackage[utf8]{inputenc}
\usepackage[a4paper,margin=1cm,landscape]{geometry}
\usepackage{tikz}
\usetikzlibrary{positioning,shapes,shadows,arrows}

\usepackage{amsrefs}

\usepackage{amsmath,amssymb,amsfonts,amsthm,amscd,amstext,amsxtra,amsopn,array,url,verbatim,mathrsfs,enumerate,anysize,enumitem,graphicx,xfrac,mathtools,xspace, comment, hyperref}

\usepackage{multirow}
\usepackage{colortbl}

\usepackage{mathtools} 
\mathtoolsset{showonlyrefs} 

\usepackage [english]{babel}
\usepackage [autostyle, english = american]{csquotes}
\MakeOuterQuote{"}

\newcommand{\imod}[1]{\,\left(\textnormal{mod }#1\right)}

\newcommand{\bigO}{\mathcal{O}}
\newcommand{\littleo}[1]{o\left(#1\right)\xspace}

\newtheorem{thm}{Theorem}[section]
\newtheorem{theorem}[thm]{Theorem}
\newtheorem{proposition}[thm]{Proposition}

\newtheorem{lemma}[thm]{Lemma}

\newtheorem*{theorem*}{Theorem}
\newtheorem*{problem*}{Problem}
\newtheorem*{corollary*}{Corollary}

\theoremstyle{definition}

\newcommand{\R}{\mathbb{R}} 
\newcommand{\Z}{\mathbb{Z}} 
\title{Explicit Improvements to the Burgess Bound via P\'{o}lya--Vinogradov}
\author{
  Matteo Bordignon\\
  \texttt{m.bordignon@student.unsw.edu.au}\\
  University of New South Wales Canberra, School of Science\\
  \\
  Forrest J. Francis\\
  \texttt{f.francis@student.unsw.edu.au}\\
  University of New South Wales Canberra, School of Science
}
\date{\today}

\begin{document}

\maketitle

\begin{abstract}
  We make explicit a theorem of Fromm and Goldmakher \cite{Fromm}, which states that one can improve Burgess' bound for short character sums simply by improving the leading constant in the P\'{o}lya--Vinogradov inequality. Towards achieving this, we establish explicit versions of several estimates related to the mean values of real multiplicative functions and the Dickman function. 
\end{abstract}

\section{Introduction}

Given a Dirichlet character $\chi \imod{q}$, it is often the case that we need to consider the size of the corresponding \emph{character sum}, 
\begin{equation}\label{charactersum}
S_\chi(t) = \sum_{n \leq t} \chi(n).
\end{equation}
Owing to the orthogonality relation on residues modulo $q$, one only ever needs to consider the case that the character sum is \emph{short}, i.e., $t \leq q$. In this case, we have the trivial estimate,
\[\lvert S_\chi(t) \rvert \leq t. \]

There are two standard non-trivial estimates for the size of \eqref{charactersum}. First, the \emph{P\'{o}lya--Vinogradov inequality},
$S_\chi(t) \ll \sqrt{q}\log{q}$ (henceforth referred to as the ``P--V inequality''). Second, \emph{Burgess' bound}, $S_\chi(t) \ll t^{1-\frac{1}{r}}q^{\frac{r+1}{4r^2}+\epsilon}$, for $\epsilon > 0$ and an integer $r > 2$.
If the modulus $q$ is a prime, then both of these estimates can be used to show that 
\begin{equation}\label{littleoburgess}
S_\chi(t) = \littleo{t},
\end{equation}
for large enough $t$. 

One might consider the Burgess bound to be a better result, however, unless the character sum is particularly long. Specifically, P--V implies that \eqref{littleoburgess} holds for $t > q^{\frac{1}{2}+\epsilon}$, while Burgess' bound implies that \eqref{littleoburgess} holds for $t > q^{\frac{1}{4}+ \littleo{1}}$. The proof of the Burgess bound also relies on advanced results due to Weil \cite{weil19481}, while the standard proof of the P--V inequality is substantially easier. Finally, one of the best-known P--V inequalities is proved using the effective range of Burgess' bound, see \cite{Hildebrand1988}.

Conversely, when working with explicit versions of these estimates, any improvement to the leading constant in the P--V inequality will immediately yield improvements in the leading constant for Burgess' bound (see, for example, \cite{trevino2015} and \cite{Forrest}). Fromm and Goldmakher \cite{Fromm} have recently established that, in fact, improvements to the P--V inequality can be used to extract improvements to the effective range (with respect to $t$) in Burgess' bound. Precisely, they establish the following relationship.
\begin{theorem}{\textnormal{\cite{Fromm}*{Theorem A}}}\label{Theo:frommA}
Suppose the P--V inequality can be improved to $S_\chi(t) =  \littleo{\sqrt{q}{\log{q}}}$ for all even primitive quadratic $\chi \imod{q}$. Then $S_\xi(t) = \littleo{t}$ for all $t \gg_\epsilon p^{\epsilon}$ for all odd primitive quadratic $\xi \imod{p}$. 
\end{theorem}
Based on a suggestion Fromm and Goldmakher made in their paper, we will prove the following explicit version of Theorem \ref{Theo:frommA}. The interested reader may also consider the work of Mangerel \cite{Mangerel}, for a different approach to the relationship between P--V and Burgess. 
\begin{theorem}
\label{Theo:exp}
Suppose the P--V inequality can be improved to \[S_{\chi}(t)\leq (c_1+\littleo{1})\sqrt{q}\log q,\] for all even primitive quadratic $\chi \imod{q}$. Then for all odd primitive quadratic characters $\xi \imod{p}$ we have $S_{\xi}(t)< c t$ for $t>p^{\epsilon (c_1, c)}$, with $\epsilon (c_1,c)=4 \pi  \frac{c_1}{\delta(c)^{3/2}}+o_t(1)$ and $\delta(c)$ as in Lemma \ref{lemma:M-L}, such that $\delta(c) \le 2/7$.  
\end{theorem} 
The above result is particularly interesting, as it the first to shows that a Burgess-like result depends in a meaningful way on the leading constant in the P--V inequality. In Table \ref{samplevalues}, we compare $\epsilon(c_1,c)$ for various $c$ using the best known P--V constant and several powers of $10$.

\begin{table}[ht]
\centering
\caption{Sample values for $\epsilon(c_1,c)$.}
\label{samplevalues}
\begin{tabular}{r|r|r|r|r|r|}
\cline{2-6}
                                                              & \multicolumn{1}{c|}{\cellcolor[HTML]{C0C0C0}$c = 0.99$}                      & \multicolumn{1}{c|}{\cellcolor[HTML]{C0C0C0}0.5}                             & \multicolumn{1}{c|}{\cellcolor[HTML]{C0C0C0}0.25}                & \multicolumn{1}{c|}{\cellcolor[HTML]{C0C0C0}0.05}                & \multicolumn{1}{c|}{\cellcolor[HTML]{C0C0C0}0.025}               \\ \hline
\rowcolor[HTML]{EFEFEF} 
\multicolumn{1}{|r|}{\cellcolor[HTML]{C0C0C0}$c_1$}           & \multicolumn{1}{c|}{\cellcolor[HTML]{EFEFEF}$\delta(c) = 1.56\cdot10^{-10}$} & \multicolumn{1}{c|}{\cellcolor[HTML]{EFEFEF}$5.51\cdot10^{-11}$} & \multicolumn{1}{c|}{\cellcolor[HTML]{EFEFEF}$1.92\cdot10^{-11}$} & \multicolumn{1}{c|}{\cellcolor[HTML]{EFEFEF}$1.65\cdot10^{-12}$} & \multicolumn{1}{c|}{\cellcolor[HTML]{EFEFEF}$5.78\cdot10^{-13}$} \\ \hline
\multicolumn{1}{|r|}{\cellcolor[HTML]{EFEFEF}1}               & $9.15\cdot10^{15}$                                                           & $4.35\cdot10^{16}$                                                           & $2.12\cdot10^{17}$                                               & $8.32\cdot10^{18}$                                               & $4.05\cdot10^{19}$                                               \\ \hline
\multicolumn{1}{|r|}{\cellcolor[HTML]{EFEFEF}$(2\pi^2)^{-1}$} & $4.64\cdot10^{14}$                                                           & $2.21\cdot10^{15}$                                                           & $1.08\cdot10^{16}$                                               & $4.22\cdot10^{17}$                                               & $2.05\cdot10^{18}$                                               \\ \hline
\multicolumn{1}{|r|}{\cellcolor[HTML]{EFEFEF}$10^{-5}$}       & $9.15\cdot10^{10}$                                                           & $4.35\cdot10^{11}$                                                           & $2.12\cdot10^{12}$                                               & $8.32\cdot10^{13}$                                               & $4.05\cdot10^{14}$                                               \\ \hline
\multicolumn{1}{|r|}{\cellcolor[HTML]{EFEFEF}$10^{-10}$}      & $9.15\cdot10^{5}$                                                            & $4.35\cdot10^{6}$                                                            & $2.12\cdot10^{7}$                                                & $8.32\cdot10^{8}$                                                & $4.05\cdot10^{9}$                                                \\ \hline
\multicolumn{1}{|r|}{\cellcolor[HTML]{EFEFEF}$10^{-15}$}      & $9.15$                                                                       & $43.5$                                                                       & $212$                                                            & $8320$                                                           & $4.05\cdot10^{4}$                                                   \\ \hline
\multicolumn{1}{|r|}{\cellcolor[HTML]{EFEFEF}$10^{-20}$}      & $8.45\cdot10^{-14}$                                                          & $4.35\cdot10^{-4}$                                                           & $2.12\cdot10^{-3}$                                               & $8.32\cdot10^{-2}$                                               & $0.405$                                                          \\ \hline
\end{tabular}
\end{table}

From Table \ref{samplevalues}, one sees that $\epsilon(c_1,c)$ roughly decays in magnitude as $c_1$ does. However, even to obtain an improvement over the trivial bound would require significant improvements over the best available choices of $c_1$. One should expect this behaviour, since one also expects to be able to take $c_1$ tending to 0. Additionally, since the best $c_1$ in the P--V inequality is obtained via Burgess' bound, one does not expect to have $\epsilon(c_1,c) < 0.25$ for all $c$ while $c_1$ is fixed. While there is room for improvement in $\epsilon (c_1,c)$, we believe that our result has significance as the first of its kind. This is also part of the reason, together with the heavy analytic machinery employed, why $\epsilon (c_1,c)$ is not yet optimal. We hope this result will increase the interest in the explicit correlation between P--V and Burgess' bound. 

As an aside, note that in Theorem \ref{Theo:exp}, we have still included some $\littleo{1}$ terms. This is because many of the best known P--V results appear in this form. This choice also makes the exposition more concise. Further attempts in line with this article, in particular those using completely explicit P--V results like \cite{frolenkov2013} or \cite{Bordignon}, should be able to make the result completely explicit.

In order to obtain Theorem \ref{Theo:exp}, we must establish some notation. Let
\begin{equation*}
    \mathbf{M}_f(x):= \frac{1}{x}\sum_{n \leq x}f(n) \quad \text{and}\quad \mathbf{L}_f(x):=\frac{1}{\log x}\sum_{n \leq x}\frac{f(n)}{n}.
\end{equation*}
The result that allowed Fromm and Goldmakher to obtain Lemma A in \cite{Fromm} is a correlation between the two functions defined above. This correlation, Lemma B in \cite{Fromm}, assures us that if $\mathbf{M}_f(x)$ is bounded away from zero, then $\mathbf{L}_f(x)$ will be as well (for certain $f$). The proof of Theorem \ref{Theo:exp} relies on establishing an explicit version of Lemma B in \cite{Fromm}.
\begin{lemma}
\label{lemma:M-L}
Given $c>0$ and $x_0=x_0(c)\ge 1$ such that 
\begin{equation*}
|\mathbf{M}_f(x)|\ge c \Rightarrow \mathbf{L}_f(x)\ge \delta(c),
\end{equation*}
with
\begin{equation*}
    \delta(c):= 0.2\exp\left(-\frac{1}{K} \log \left(\frac{9.75\cdot 10^{5}}{c}\right)\left(1.42\left(\frac{9.75\cdot 10^{5}}{c}\right)^{\frac{1}{2K}}+1/2\right)\right)+o_{x}(1),
\end{equation*}
for all completely multiplicative functions $f:\mathbb{Z}\rightarrow [-1,1]$, $x> x_0$, $K\approx 0.3286$.
\end{lemma}
This result allows us to prove Theorem \ref{Theo:exp}.
\begin{proof}[Proof of Theorem \ref{Theo:exp}]
Here we follow the proof of Theorem A \cite{Fromm}. Using Lemma 2.1 \cite{Fromm} and assuming $|\mathbf{M}_{\xi}(x)|\ge c$, we obtain infinitely many characters $\chi$ such that
\begin{equation*}
|S_{\chi}(N)|\ge \left(\frac{\sqrt{l}\delta(c)\epsilon}{2\pi\varphi(l)}+\littleo{1}\right)\sqrt{q}\log q,
\end{equation*}
with $l$ the least prime larger than $\frac{2}{\delta(c)}$ which satisfies $l \equiv 3 \pmod{4}$. We therefore have a contradiction if $\frac{\sqrt{l}\delta(c)\epsilon}{2\pi \varphi(l)} > c_1$, i.e. when $\epsilon > 2 \pi c_1 \frac{\varphi(l)}{\sqrt{l}\delta(c)}$. We can further simplify this by observing that we trivially have $\varphi(l)\le l$, that results optimal for large $l$. Using the version of Bertrand's postulate for primes in arithmetic progressions in \cite{Breusch}, with the assumption $\frac{2}{\delta(c)} \ge 7$, we have that $l\leq \frac{4}{\delta(c)}$. Note that assuming a smaller upper bound for $\delta(c)$, together with Corollary 6 in \cite{Bennett}, is possible to reduce the constant $4$ to $2+o(1)$, we decided not to do so to keep the result as concise as possible. Thus, we obtain
\begin{equation*}
\epsilon > 4 \pi  \frac{c_1}{\delta(c)^{\frac{3}{2}}}.
\end{equation*} 
\end{proof}
The proof of Lemma \ref{lemma:M-L} will require two results, which will make up the bulk of this article.
The easier of these is the following explicit version of Theorem 2 in \cite{Hildebrand} applied to $(1*f)(n)$ (another non-explicit version of this result can be found in \cite{Granville2007.1}). First, for a given multiplicative function $f$, let us define
\begin{equation*}
u:=\sum_{p\leq x}\frac{1-f(p)}{p}.
\end{equation*}
\begin{theorem}
\label{theo:lower}
Let $f(x)$ be a completely multiplicative function as defined in (1.1) of \cite{Hildebrand}. Then, we have
\begin{equation*}
\frac{1}{x}\sum_{n\leq x} (1*f)(n) \ge \left(0.2+\littleo{1}\right)\log x~e^{-u\left(1.42e^{\frac{u}{2}}+\frac{1}{2}\right)}+\littleo{1}.
\end{equation*}
\end{theorem}
The second result, which is the harder to prove, is an explicit version of Theorem III.4.14 in Hall and Tenenbaum \cite{tenenbaum2015}. In our current application, we focus on functions $g(n)$ which are quadratic Dirichlet characters, but there are variants of this theorem which cover a much larger class of functions (for example, see the main theorem of \cite{hall1991}). 

\begin{theorem}\label{thm:forresttheorem}
Let $K$ be the unique solution to 
\[\frac{1}{2\pi} \int_0^{2\pi} \! \lvert \cos(t) - K \rvert  \, dt  = 1 - K.\]
Note that $K \approx 0.3286$. If $f$ is a real, completely multiplicative function, we have, uniformly for $x \geq 1$,
\begin{equation}\label{tenenbaumexplicit}
|\mathbf{M}_f(x)| \leq (9.75\cdot 10^{5}+\littleo{1})\exp\left\{-K\sum_{p\leq x}\frac{1-f(p)}{p}\right\}+\littleo{1}.
\end{equation} 
 
\end{theorem}
We can now easily prove Lemma \ref{lemma:M-L}.
\begin{proof}[Proof of Lemma \ref{lemma:M-L}]
Here, we follow the proof of Lemma B in \cite{Fromm}. 
Theorem \ref{thm:forresttheorem} gives
\begin{equation}
\label{eq:lm1}
    \left(\frac{|\mathbf{M}_f(x)|+\littleo{1}}{9.75\cdot 10^{5}+\littleo{1}}\right)^{-\frac{1}{2K}}\ge e^{\frac{u}{2}} \quad \text{and} \quad 
  \frac{1}{K} \log \left(\frac{|\mathbf{M}_f(x)|+\littleo{1}}{ 9.75\cdot 10^{5}+\littleo{1}}\right)^{-1}\ge u.
\end{equation}
It is easy to see that
\begin{equation*}
    \mathbf{L}_f(x)=\frac{1}{x\log x}\sum_{n\ge x}(1*f)(n)+\littleo{1},
\end{equation*}
and, by Theorem \ref{theo:lower}, we obtain 
\begin{align}
\label{eq:1}
\mathbf{L}_f(x) \ge \left(0.2+\littleo{1}\right) e^{-u(1.42e^{u/2}+1/2)}\log x+\littleo{1}.
\end{align}
The result follows substituting \eqref{eq:lm1} in \eqref{eq:1} and remembering that $|\mathbf{M}_f(x)|\ge c$.
\end{proof}
In Section \ref{sec:Dickman} we will prove Theorem \ref{theo:lower}. In Section \ref{sec:parexp} we will prove a partially explicit version of an upper bound for the mean value of multiplicative functions, that works as an intermediate result for Theorem \ref{thm:forresttheorem}. In Section \ref{sec:last}, we introduce some explicit bounds related to prime numbers; applying these results to those obtained in the previous sections, we conclude with a proof of Theorem \ref{thm:forresttheorem}. To ease the understanding of the relationships between the results we introduce the following scheme.
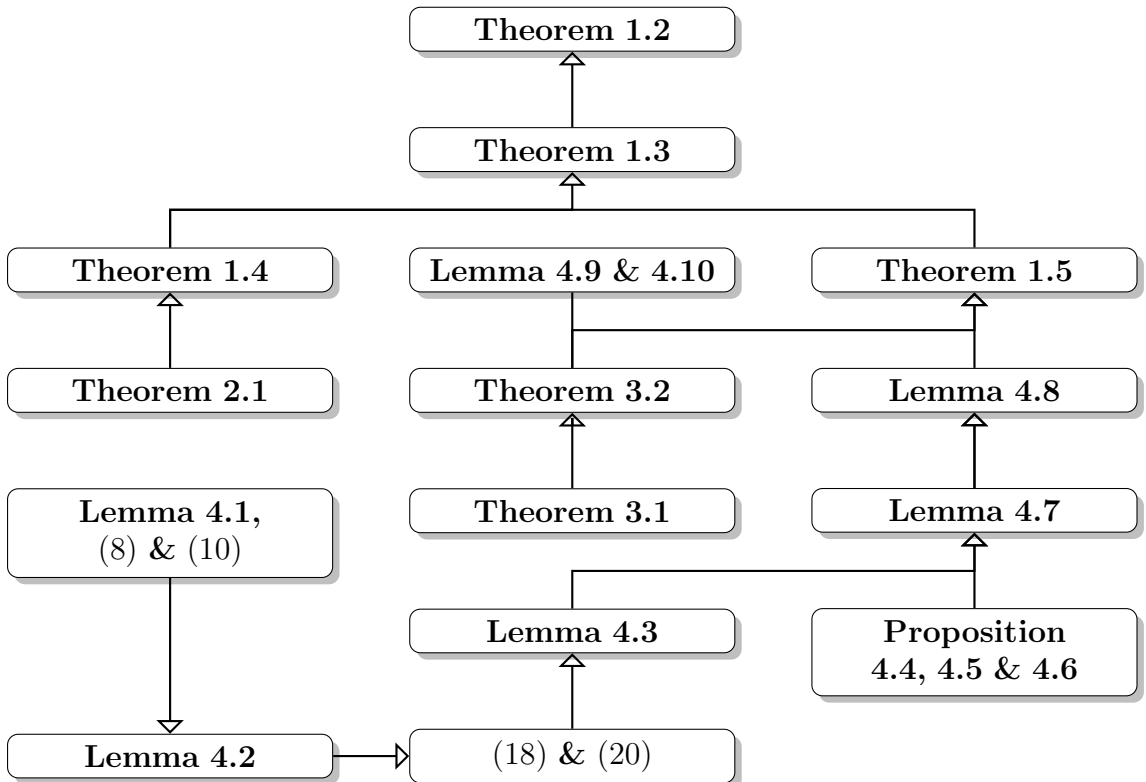
\tikzstyle{abstract}=[rectangle, draw=black, rounded corners, fill=white, drop shadow,
        text centered, anchor=north, text=black, text width=4cm]
\tikzstyle{myarrow}=[->, >=open triangle 90, thick]
\tikzstyle{line}=[-, thick]

\begin{center}
\begin{tikzpicture}[node distance=1cm]
    \node (Theorem13) [abstract, rectangle split, rectangle split parts=1]
        {
            \textbf{Theorem \ref{Theo:exp}}
        };
    \node (Theorem12) [abstract, rectangle split, rectangle split parts=1, below=of Theorem13]
        {
            \textbf{Theorem \ref{lemma:M-L}}
        };
        \draw[myarrow] (Theorem12.north) -- ++(0,0.8) -|(Theorem13.south);
    \node (Lemma4910) [abstract, rectangle split, rectangle split parts=1, below=of Theorem12]
        {
            \textbf{Lemma \ref{Fexplicit} \& \ref{lemma:v_1}}
            \nodepart{second}
        };
    \node (Theorem14) [abstract, rectangle split, rectangle split parts=1, left=of Lemma4910]
        {
            \textbf{Theorem \ref{theo:lower}}
            \nodepart{second}
        };
    \node (Theorem15) [abstract, rectangle split, rectangle split parts=1, right=of Lemma4910]
        {
            \textbf{Theorem \ref{thm:forresttheorem}}
            \nodepart{second}
        };
    \draw[myarrow] (Theorem14.north) -- ++(0,0.5) -| (Theorem12.south);
    \draw[line] (Theorem14.north) -- ++(0,0.5) -| (Theorem15.north);  
    
     \node (Theorem21) [abstract, rectangle split, rectangle split parts=1, below=of Theorem14]
        {
            \textbf{Theorem \ref{lemma:dick}}
            \nodepart{second}
        };
        \draw[myarrow] (Theorem21.north) -- ++(0,0.8) -| (Theorem14.south);
         \node (Theorem32) [abstract, rectangle split, rectangle split parts=1, below=of Lemma4910]
        {
            \textbf{Theorem \ref{thm:matteotheorem}}
            \nodepart{second}
        };
         \draw[line] (Lemma4910.south) |-  (Theorem32.north);  
        \node (Theorem31) [abstract, rectangle split, rectangle split parts=1, below=of Theorem32]
        {
            \textbf{Theorem \ref{lemma:zeta1}}
            \nodepart{second}
        };
        \draw[myarrow] (Theorem31.north) -- ++(0,0.93) -| (Theorem32.south);
        \draw[myarrow] (Theorem32.north) -- ++(0,0.5) -| (Theorem15.south);
      \node (Lemma48) [abstract, rectangle split, rectangle split parts=1, below=of Theorem15]
        {
            \textbf{Lemma \ref{lambdalemma}}
            \nodepart{second}
        };
        \draw[myarrow] (Lemma48.north) -- ++(0,0.8) -| (Theorem15.south);
     \node (Lemma47) [abstract, rectangle split, rectangle split parts=1, below=of Lemma48]
        {
            \textbf{Lemma \ref{lemma:explicitbigconst}}
            \nodepart{second}
        };
        \draw[myarrow] (Lemma47.north) -- ++(0,0.8) -| (Lemma48.south);
    \node (Proposition446) [abstract, rectangle split, rectangle split parts=1, below=of Lemma47]
        {
            \textbf{Proposition \ref{mertensexplicit}, \ref{mertensexplicit1} \& \ref{lsquaresum}}
            \nodepart{second}
        };
        \draw[myarrow] (Proposition446.north) -- ++(0,0.8) -| (Lemma47.south);
        \draw[myarrow] (Lemma47.north) -- ++(0,0.8) -| (Lemma48.south);
    \node (Lemma43) [abstract, rectangle split, rectangle split parts=1, below=of Theorem31]
        {
            \textbf{Lemma \ref{boundedvariationest}}
            \nodepart{second}
        };
        \draw[myarrow] (Lemma43.north) -- ++(0,0.5) -| (Lemma47.south);
        \node (1921) [abstract, rectangle split, rectangle split parts=1, below=of Lemma43]
        {
            \textbf{\eqref{eq:R31} \& \eqref{eq:R32}}
            \nodepart{second}
        };
        \draw[myarrow] (1921.north)  -- (Lemma43.south);
        \node (Lemma42) [abstract, rectangle split, rectangle split parts=1, left=of 1921]
        {
            \textbf{Lemma \ref{lemma:piup}}
            \nodepart{second}
        };
        \draw [myarrow]  (Lemma42) -- (1921);
        \node (Lemma41) [abstract, rectangle split, rectangle split parts=1, below=of Theorem21]
        {
            \textbf{Lemma \ref{lemma:lilow}, \eqref{pi} \& \eqref{eq:piexp}}
            \nodepart{second}
        };
        \draw[myarrow] (Lemma41.south)  -- (Lemma42.north);

\end{tikzpicture}
\end{center}
\section{Lower bound for the mean value theorem for a non-negative multiplicative function }
\label{sec:Dickman}
The aim of this section is to prove Theorem \ref{theo:lower}.
We start by giving an explicit lower bound for the Dickman function, $\rho(x)$, defined by  
\begin{equation*}
    x\rho '(x)+\rho (x-1)=0,
\end{equation*}
with initial conditions $\rho (x)=1$ for $0\leq x \leq 1$. Note that we will follow Buchstab's approach from \cite{buchstab} for large $x$, alongside computations for small $x$.
\begin{lemma}
\label{lemma:dick}
Assuming $x\ge 1$, we have
\begin{equation}
\label{eq:df}
\rho(x)\ge x^{-1.42x}.
\end{equation}
\end{lemma}
\begin{proof}
Using the the built-in Dickman function in Sage, we determine that for $1 \leq x \leq 130$ we can take as an exponent $1.15$. Note that we are limited to this interval due to the computational complexity.  We can thus use the following result due to Buchstab \cite{buchstab}, that tells that for $x\ge 6$ and $\delta=\frac{1}{\log x +1+\frac{\log x}{x}}< \frac{1}{3}$, we have
\begin{equation}
\label{eq:df1}
\rho(x)\ge \exp\left(-x\left(1+\frac{1}{\log x}\right)\left(\log (x+\delta)+\log \frac{1}{\delta} -1\right)-2\log x\right),  
\end{equation}
and the result follows taking $x\ge 130$.
\end{proof}
It worth noting that, by \cite{buchstab}, the right size for the constant in the exponent of \eqref{eq:df} is $1+o(1)$. Since we want a uniform result, a lower bound for the $1+o(1)$ term appears, by computation, to be $1.15$. Obtaining this result appears difficult as \eqref{eq:df1} does not give a good estimate for small values of $x$. One might get around this by making explicit other asymptotic results for $\rho(x)$, such as the one in \cite{deBruijn}, but we have not pursued this here.
We can now prove Theorem \ref{theo:lower}.
\begin{proof}[Proof of Theorem \ref{theo:lower}]
This is Theorem 2 \cite{Hildebrand} with $K=2$, $K_2=1.1$ and $z=2$, used together with $\max \left(0,1-(1*f)(p)\right)\leq \frac{1-f(p)}{2}$ and Lemma \ref{lemma:dick}. We also need to note that
\begin{equation*}
\prod_{p\leq x} \left(1-\frac{1}{p}\right)\left( 1+\frac{(1*f)(p)}{p}+\frac{(1*f)(p^2)}{p^2}\cdots\right)=\prod_{p\leq x}\frac{1-\frac{1}{p}}{1-\frac{(1*f)(p)}{p}}
\end{equation*}
\begin{equation*}
 \ge e^{-u}\exp\left(\sum_{p\leq x} \frac{1}{p}\right)\exp\left(\sum_{p\leq x}\left(\log \left(1-\frac{1}{p}\right)+\frac{1}{p}\right)\right).
\end{equation*}
We can conclude using Theorem 5 and Corollary 1 \cite{rosser1962}, that gives
\begin{equation*}
\exp\left(\sum_{p\leq x}\left(\log \left(1-\frac{1}{p}\right)+\frac{1}{p}\right)\right)= \exp(M-\gamma)\ge \exp(-0.32),
\end{equation*}
with $M$ the Meissel--Mertens constant and $\gamma$ the Euler--Mascheroni constant.
\end{proof}

\section{A partially explicit upper bound for the mean value of multiplicative functions}
\label{sec:parexp}
In this section, we aim to prove an explicit version of a theorem of Montgomery \cite{montgomery1978}, regarding the mean value of multiplicative functions. He restricted his interest, as will we, to completely multiplicative functions. The more general case involves technical changes, see \cite{tenenbaum2015}, which make the leading constant increase significantly.

We start by introducing a well-known, but useful, result.
\begin{lemma}
\label{lemma:zeta1}
Assuming $ s=1+\alpha +i\tau$, with $\alpha \searrow 0$ and $|\tau|\leq 1/2$ we have
\begin{equation*}
\left|\frac{\zeta'}{\zeta}(s)\right|\leq \frac{1}{|s-1|}+\bigO(1). 
\end{equation*}
\end{lemma}
\begin{proof}
By Euler--Maclaurin, we have
\begin{align*}
\sum_{n\leq  N} \frac{1}{n^s}=\int_1^{N} \frac{1}{x^s}dx +\frac{1}{2}\left(\frac{1}{N^s}+1\right)-s\int_1^{N} \frac{1}{x^{s+1}} \left(\{x\}-\frac{1}{2}\right)dx.
\end{align*}
Thus, taking $N\rightarrow \infty$,
\begin{align*}
\left|\sum_{n}^{\infty} \frac{1}{n^s}-\int_1^{\infty} \frac{1}{x^s}dx\right|\leq  \frac{1}{2}(1+|s|).
\end{align*}
Now it follows from
\begin{align*}
\int_1^{\infty} \frac{(\log x)^{\ell}}{x^s}dx =\frac{\ell!}{(s-1)^{\ell+1}},
\end{align*}
that
\begin{align*}
\left|\zeta(s)-\frac{1}{(s-1)}\right|\leq  \frac{1}{2}(1+|s|).
\end{align*}
Proceeding in the same way, we obtain
\begin{align*}
\left|\zeta'(s)+\frac{1}{(s-1)^2}\right|\leq  \frac{1}{2}(1+|s|).
\end{align*}
The result easily follows remembering that $\alpha\searrow 0$.
\end{proof}
Everything is in place to prove an explicit version of the inequality in \cite{montgomery1978}. Note that our result appears slightly different with compared with the cited one, as we have tailored the optimization of the constant for the current application.
\begin{theorem}
\label{thm:matteotheorem}
Let $g$ be a completely multiplicative function such that $|g(n)|\leq 1$. Set
\begin{equation*}
G(x):=\sum_{n \leq x}g(n), \;\;\;\;\;\; F(s):=\sum_{n=1}^{\infty} g(n) n^{-s}.
\end{equation*}
We define
\begin{equation*}
H(\alpha)^2:=\sum_{\substack{k\in \mathbb{Z} }}\frac{1}{(k-1/2)^2+1}\max_{\substack{\sigma=1+\alpha \\ |\tau-k|\leq \frac{1}{2}}}|F(s)|^2.
\end{equation*}
Then, for $x\ge x_0$ large enough, 
\begin{equation}
\label{eq:Gexp}
G(x)\leq \left( 3.14+\littleo{1}\right) \frac{x}{\log x}\int_{1/\log x}^1H(\alpha)\frac{d\alpha}{\alpha}+O_{x_0}\left(\frac{x}{\sqrt{\log x}}\right).
\end{equation}
\end{theorem}
\begin{proof}
We now establish, for $x\ge x_0$, the following result 
\begin{equation}
\label{eq:Gexp2}
\int_{\sqrt{x}}^x\frac{|G(t)|}{t^2}dt \leq \left( \sqrt{\frac{ 9.45}{2}}+\littleo{1}\right) H\left( \frac{2}{\log x}\right)+O\left(\sqrt{\log x}\right).
\end{equation}
 By the Cauchy--Schwarz inequality, with $\alpha =2/\log x$,
\begin{equation*}
\int_{\sqrt{x}}^x \frac{|G(t)|}{t^2}dt \leq \left(\int_1^x \frac{ \left(|G(t)|\log t\right)^2 }{t^{3+2\alpha}}dt  \int_{\sqrt{x}}^x \frac{1}{\log^2 t ~t^{1-2\alpha}}dt  \right)^{1/2}.
\end{equation*}
We can observe that, with $n\in \mathbb{N}$,
\begin{equation*}
\int_{\sqrt{x}}^x \frac{1}{\log^2 t ~t^{1-2\alpha}}dt\leq \sum_{j=0}^{n-1}\int_{x^{\frac{1}{2}\left(\frac{j}{n}+1\right)} }^{x^{\frac{1}{2}\left(\frac{j+1}{n}+1\right)} } \frac{1}{\log^2 t ~t^{1-2\alpha}}dt
\end{equation*}
\begin{equation*}
\leq \frac{1}{\log^2 x}\sum_{j=0}^{n-1}\frac{4}{\left(\frac{j}{n}+1\right)^2}\int_{x^{\frac{1}{2}\left(\frac{j}{n}+1\right)} }^{x^{\frac{1}{2}\left(\frac{j+1}{n}+1\right)} } \frac{1}{ ~t^{1-2\alpha}}dt  \leq\frac{1}{\log^2 x~\alpha}(e^{\frac{2}{n}}-1)2e^2\sum_{j=0}^{n-1}\frac{e^{2j/n}}{(\frac{j}{n}+1)^2}
\end{equation*}
\begin{equation*}
\leq  \frac{1}{\log^2 x~\alpha}(e^{\frac{2}{n}}-1)2n\int_{1}^{2}\frac{e^{2y}}{y^2}dy\leq \frac{4\cdot 9.45}{\log^2 x~\alpha} .
\end{equation*}
Defining $K(t):=\sum_{n\leq t}g(n)\log n$, then
\begin{equation*}
G(t)\log t-K(t) \ll t.
\end{equation*}
Thus, taking $\alpha=2/\log x$, the proof of \eqref{eq:Gexp2} reduces to that of
\begin{equation*}
\int_{1}^x\frac{|K(t)|^2}{t^{3+2\alpha}}dt \leq \left(\frac{1}{2}+\littleo{1}\right) \frac{H(\alpha)^2}{\alpha}.
\end{equation*}
The equation
\begin{equation*}
\int_0^{\infty} K(e^u)e^{-u\sigma} e^{-iur}du=\frac{-F'(s)}{s}\;\; (\sigma> 1),
\end{equation*}
allows us to write Plancherel's formula as
\begin{equation*}
\int_{1}^x\frac{|K(t)|^2}{t^{3+2\alpha}}dt =\frac{1}{2\pi}\int_{-\infty}^{\infty}\left|\frac{F'(1+\alpha+i\tau)}{1+\alpha+i\tau} \right|^2 d\tau.
\end{equation*}
We assume $T$ arbitrary large. For $|\tau|> T$ we have, by (4.46) in \cite{tenenbaum2015}
\begin{equation*}
\int_{|\tau|> T}\left|\frac{F'(1+\alpha+i\tau)}{1+\alpha+i\tau} \right|^2 d\tau \ll \frac{1}{T}+\frac{1}{\alpha^3T^2}.
\end{equation*}
We now need to estimate the contribution in the complementary range $|\tau|\leq T$. We write
\begin{equation*}
\label{eq:Kexp2}
\int_{|\tau|\leq T}\left|\frac{F'(1+\alpha+i\tau)}{1+\alpha+i\tau} \right|^2 d\tau 
\end{equation*}
\begin{equation*}\leq \sum_{|k| \leq T}\frac{1}{1+(k-1/2)^2}\int_{k-1/2}^{k+1/2}\left|F'(1+\alpha+i\tau) \right|^2 d\tau.
\end{equation*}
The right hand side integral does not exceed
\begin{equation*}
\max_{\substack{\sigma=1+\alpha \\ |\tau-k|\leq 1/2}}|F(s)|^2\int_{k-1/2}^{k+1/2}\left|\frac{F'}{F}(1+\alpha+i\tau) \right|^2 d\tau.
\end{equation*}
We can observe that
\begin{equation*}
\left| \frac{F'}{F}(s)\right|^2\leq \left| \frac{\zeta'}{\zeta}(s)\right|^2,
\end{equation*}
and, choosing $x \ge x_0$ to have $\alpha =2/\log x$ small enough, by Lemma \ref{lemma:zeta1} we obtain
\begin{align*}
\int_{k-1/2}^{k+1/2}\left|\frac{F'}{F}(1+\alpha+i\tau) \right|^2 d\tau & \leq \int_{-1/2}^{1/2}\left|\frac{\zeta'}{\zeta}(1+\alpha+i\tau) \right|^2 d\tau  \\ &\leq  \int_{-1/2}^{1/2}\frac{1}{\alpha^2+\tau^2}d\tau+\bigO(1)=\frac{\pi}{\alpha}+\bigO(1).
\end{align*}
Thus \eqref{eq:Gexp2} is obtained taking $T\rightarrow \infty$.
We now introduce (4.39) from \cite{tenenbaum2015}
\begin{equation}
\label{eq:Gexp1}
|G(x)| \leq \frac{x}{\log x}\int_{1}^x\frac{|G(t)|}{t^2}dt+O\left(\frac{x}{\log x}\right).
\end{equation}
With the above result and using \eqref{eq:Gexp2} we can now finish the proof as follows
\begin{align*}
\int_{e^2}^x \frac{|G(t)|}{t^2}dt &\leq \frac{1}{\log 2}\int_{e^2}^x \frac{|G(t)|}{t^2}\int_t^{t^2}\frac{dy}{y\log y}dt\\&\leq  \frac{1}{\log 2}\int_{e^2}^{x^2}\frac{dy}{y\log y}\int_{\sqrt{y}}^y \frac{|G(t)|}{t^2}dt  \\& \leq \frac{1}{\log 2}\left( \sqrt{\frac{ 9.45}{2}}+\littleo{1}\right)\int_{1/\log x}^1 \frac{H\left( \alpha\right)}{\alpha} d\alpha +O_{x_0}\left(\sqrt{\log x}\right).
\end{align*}
\end{proof}
\section{Explicit mean value estimates for real multiplicative functions}
\label{sec:last}
In this section we aim to prove Theorem \ref{thm:forresttheorem}. We will first, in subsection \ref{sub:1} and \ref{sub:2}, introduce some useful explicit results and then tackle Theorem \ref{thm:forresttheorem} in subsection \ref{sub:3}.

\subsection{Prime counting estimates} \label{sub:1}
Take $\pi(x)$ to be the prime counting function. We provide two versions of the Prime Number Theorem (PNT), the first good for small $x$ and the second for big. Assuming $x \geq 59$, by \cite{rosser1962} we have
\begin{equation}
\label{pi}
\frac{x}{\log x}\left(1+\frac{1}{2 \log x}\right)\leq \pi(x) \leq \frac{x}{\log x}\left(1+\frac{3}{2 \log x}\right).
\end{equation}
Defining
\begin{equation}
\label{eq:li}
\textrm{li}(x)=\int_0^x \frac{1}{\ln y} dy
\end{equation} 
and taking $x\ge 229$, by Corollary 2 \cite{Trudgian}, we have
\begin{align}
\label{eq:piexp}
|\pi(x)-\textrm{li}(x)|\leq  x\frac{0.2795}{(\log x)^{3/4}}\exp\left(-\sqrt{\frac{\log x}{6.455}}\right).
\end{align}
Note that there is a better version of the PNT due to Platt and Trudgian \cite{Platt2019}. However, we will turn the above result into a uniform one and the improvement obtained using Platt and Trudgian's result is not clear and would make the following exposition longer and more complicated. We also note that another way to improve the result could be using the improved zero-free region for the Riemann zeta function given in \cite{Mossinghoff}.
We now provide some useful bounds on $\textrm{li}(x)$.
\begin{lemma}
\label{lemma:lilow}
For $x\ge 2$ we have
\begin{equation*}
\textrm{li}(x)\ge \frac{x}{\log x}\left(1+\frac{1}{\log x}\right).
\end{equation*}
\end{lemma}
\begin{proof}
By repeatedly integrating \eqref{eq:li} by parts, we have
\begin{equation*}
\textrm{li}(x)=\frac{x}{\log x}+\frac{x}{\log^2 x}+\int_0^x \frac{2}{\ln^3 y} dy,
\end{equation*}
and the result follows by observing that the last integral is positive for $x\ge 2$.
\end{proof}
From Lemma 5.9 \cite{Bennett} we have, for $x\ge 1865$ 
\begin{equation}
\label{eq:liupp}
\textrm{li}(x)\leq \frac{x}{\log x}\left(1+\frac{3}{2\log x}\right)+\textrm{li}(2).
\end{equation}

We can now prove the main lemma.
\begin{lemma}
\label{lemma:piup}
For all $x\ge 2$ we have
\begin{align}
\label{eq:pi1}
|\pi(x)-\textrm{li}(x)|\leq 0.4897 \frac{x}{\log x},
\end{align}
\begin{align}
\label{eq:pi2}
|\pi(x)-\textrm{li}(x)|\leq 1.3597 \frac{x}{\log^2 x},
\end{align}
and 
\begin{align}
\label{eq:pi3}
|\pi(x)-\textrm{li}(x)|\leq 0.1522 x \exp\left(-\sqrt{\frac{\log x}{6.455}}\right).
\end{align}
\end{lemma}
\begin{proof}
We first prove \eqref{eq:pi1} and \eqref{eq:pi2}. For $2\leq x \leq 10^5$, we obtain the result by computation.
Assuming $x\ge 10^5$, we obtain the result using \eqref{pi}, Lemma \eqref{lemma:lilow}, and \eqref{eq:liupp}. Now, we prove \eqref{eq:pi3}. For $2\leq x \leq 10^3$, we obtain the result by computation.
Assuming $x\ge 10^3$ we obtain the result using \eqref{eq:piexp}.
\end{proof}
Let $f$ be a $2\pi$-periodic function of bounded variation on $\left[0,2\pi\right]$,
writing $S(f) := \sup_t \lvert f(t) \rvert$, $V(f) := \int_0^{2\pi} \! \left| d\{f(t)\}\right|$, we can now prove the following results.
Assuming $w>1$, by \eqref{eq:pi1} we obtain
\begin{align}
\label{eq:R1}
\left| \frac{|\pi(x)-\textrm{li}(x)|f(\tau \log t)}{t}\right|_w^z\leq \frac{0.9794}{\log w} S(f).
\end{align}
and, by \eqref{eq:pi2},
\begin{align}
\label{eq:R22}
\left| \int_w^z\frac{|\pi(x)-\textrm{li}(x)|f(\tau \log t)}{t^2}dt\right|\leq \frac{1.3597}{\log w} S(f).
\end{align}
By \eqref{eq:pi3} we obtain
\begin{equation} \label{eq:Rsum}
   \left| \int_w^z \! \frac{R(t)}{t} d\{f(\tau \log{t})\} \right| \leq 0.1522 \int_{\tau \log{w}}^{\tau \log{z}} \exp\left(-\sqrt{\frac{v}{6.455\tau}}\right) \, \left|df(v) \right| 
 \end{equation}
 \begin{equation*}
\leq 0.1522 \int_{\tau \log{w}}^{\tau \log{w} + 2\pi} \! \nonumber \sum_{k=0}^\infty \exp\left(-\sqrt{\frac{v+2\pi k}{6.455\tau}}\right) \,\left|df(v) \right| 
\end{equation*}
\begin{equation*}
\leq  0.1522 V(f)  \sum_{k=0}^{\infty} \exp\left( -\sqrt{\frac{\tau \log w+2\pi k}{6.455\tau}}\right).
\end{equation*}
We focus on the two following cases.  For $0< \tau \leq 1$, $w=\exp(\frac{c}{\tau})$, with $c\ge 1$, and
\begin{equation*}
l(c,\tau,x):=\exp\left( -\sqrt{\frac{c+2\pi k}{6.455\tau}}\right),
\end{equation*}
we obtain
\begin{equation}
\label{eq:R31}
\sum_{k=0}^{\infty} l(c, \tau,k)
\leq \sum_{k=0}^{k_1} l(c, \tau,k) +\int_{k_1}^{\infty} l(c, \tau,x) dx 
\end{equation}
\begin{equation*}
\leq \sum_{k=0}^{k_1} l(c, \tau,k) + l(c, \tau,k_1)\frac{\sqrt{6.455 }\sqrt{2\pi k_1+c}+6.455 }{\pi}= O_{k_1,c}(1),
\end{equation*}
where the $O_{k_1, c}(1)$ will be computed later, optimizing on $k_1$ and $c$.
 For $\tau \ge 1$, 
 \begin{equation}
 \label{eq:wcomp}
 \log w=(1+\epsilon)6.455\log^2(\tau+3),
 \end{equation}
  \begin{equation*}
 h(\epsilon, \tau, x):=\exp\left( -\sqrt{(1+\epsilon) \log^2(\tau+3) +\frac{2\pi x}{6.455\tau}}\right),
 \end{equation*}
 with $\epsilon >0$, we obtain 
\begin{equation}
\label{eq:R32}
\sum_{k=0}^{\infty}h(\epsilon, \tau, k)
\leq\sum_{k=0}^{k_2} h(\epsilon, \tau, k)
+\int_{k_2}^{\infty}h(\epsilon, \tau, x) dx \leq 
\end{equation}
\begin{equation*}
 \sum_{k=0}^{k_2} h(\epsilon, \tau, k) +h(\epsilon, \tau, k_2)
\cdot \frac{\sqrt{6.455 \tau}\sqrt{2\pi k_2+\tau(1+\epsilon)6.455\log^2(\tau+3)}+6.455 \tau}{\pi}=O_{\epsilon, k_2}(1),
\end{equation*}
where the $O_{\epsilon, k_2}(1)$ will be computed later optimizing on $k_2$ and $\epsilon$.\\
The above upper bounds \eqref{eq:R31} and \eqref{eq:R32} will be used in the next section to prove an explicit version of Lemma III.4.13 of \cite{tenenbaum2015}. It is interesting to note that within this non-explicit result, a stronger version of \eqref{eq:piexp} was used, to assure that \eqref{eq:R31} and \eqref{eq:R32} would converge for any $w\geq 0$. As there is no explicit version of this stronger PNT, we have that the two series converge only for certain values of $w$. This will come with a loss in a term in Lemma \ref{lemma:explicitbigconst}, and therefore balancing it with the above sums will be fundamental.
\subsection{Some useful lemmas} \label{sub:2}
The bulk of the proof of Theorem \ref{thm:forresttheorem} can be contained in the following lemmas, which encapsulate explicit versions of Lemma III.4.13 of \cite{tenenbaum2015}. 

\begin{lemma}\label{boundedvariationest}
Let $f$ be a $2\pi$-periodic function of bounded variation on $\left[0,2\pi\right]$ with mean value 
\[\overline{f} := \frac{1}{2\pi}\int_0^{2\pi} \! f(t) \, dt.\]
For all real numbers $\tau$, $w$, $z$ such that $\tau \neq 0$, $1 < w < z$, we have 
\begin{equation}\label{annoyinglemmaexplicit}
      \sum_{w < p \leq z} \frac{1}{p}f(\tau \log{p}) = \overline{f} \log\left(\frac{\log{z}}{\log{w}}\right) + E_\tau(w),
  \end{equation}
where, writing $S(f) := \sup_t \lvert f(t) \rvert$, $V(f) := \int_0^{2\pi} \! \left| d\{f(t)\}\right|$.
For $0< |\tau |\leq 1$, $w=\exp(\frac{c}{\tau})$
\begin{equation}\label{aeq:err1}
\left| E_\tau(w) \right| \leq \left( \frac{\pi}{2c}+0.1522O_{k_1, c}(1)\right) V(f) +  \frac{2.3391}{c} S(f),
\end{equation}
with $O_{k_1, c}(1)$ defined in \eqref{eq:R31}, while for $|\tau| \ge 1$, $w=\exp((1+\epsilon)6.455\log^2(\tau+3))$, with~$\epsilon >0$, 
\begin{equation}\label{eq:err2}
\left| E_\tau(w) \right| \leq \left(\frac{\pi}{2}\frac{ 1}{\tau\log{w}} +0.1522O_{k_2,\epsilon}(1) \right)V(f)+ \frac{2.3391}{(1+\epsilon)6.455\log^2(\tau+3)}   S(f),
\end{equation}
with $O_{k_2,\epsilon}(1)$ defined in \eqref{eq:R32}.
\end{lemma}
\begin{proof}
It is sufficient to prove this for $\tau > 0$. Define $ R(t) := \pi(t) - \textrm{li}(t) $.
By partial summation, we have
\begin{equation}\label{longpartialsummation}
    \begin{aligned}
     \sum_{w < p \leq z} \frac{1}{p}f(\tau \log{p}) &= \int_w^z \! \frac{f(\tau\log{t})}{t\log{t}} \, dt + \left.\frac{R(t) f(\tau\log{t})}{t}\right|_w^z - \int_w^z \! R(t) \, d\left(\frac{f(\tau \log{t})}{t}\right) \\ &= \overline{f} \log\left(\frac{\log{z}}{\log{w}}\right) + \int_{\tau \log{w}}^{\tau \log{z}} \! \left(f(t) - \overline{f}\right) \frac{dt}{t} + \left.\frac{R(t) f(\tau\log{t})}{t}\right|_w^z \\ &- \int_w^z \! \frac{R(t)}{t} d\{f(\tau \log{t})\} + \int_w^z \! \frac{R(t)f(\tau\log{t})}{t^2}dt.
    \end{aligned}
\end{equation}
For the second term in \eqref{longpartialsummation}, we have from Equation 3.6 of \cite{hall1988} that, for any real $a$ and $b$,
\begin{equation}\label{variationbound}
\left|  \int_a^b \! \left(f(t) - \overline{f}\right) \, dt\right| \leq \frac{\pi}{4}V(f).
\end{equation}
By the second mean value theorem for integrals, there exists a $c \in (\tau \log{w}, \tau\log{z}]$ so that
\begin{equation}\label{mvt}
   \int_{\tau \log{w}}^{\tau \log{z}} \! \left(f(t) - \overline{f}\right) \frac{dt}{t}  = \frac{1}{\tau\log{w}}\int_{\tau\log{w}}^c (f(t) - \overline{f}) dt +  \frac{1}{\tau\log{z}}\int_c^{\tau\log{z}} (f(t) - \overline{f}) dt.
\end{equation}

Combining \eqref{variationbound} and \eqref{mvt}, we determine
\begin{equation}\label{estimate3}
\left| \int_{\tau \log{w}}^{\tau \log{z}} \! \left(f(t) - \overline{f}\right) \frac{dt}{t} \right| \leq \frac{\pi}{2}\frac{ V(f)}{\tau\log{w}}. 
\end{equation}
The third term was previously estimated in \eqref{eq:R1}, the fourth in \eqref{eq:Rsum}, \eqref{eq:R31} and \eqref{eq:R32}, and the fifth in \eqref{eq:R22}.
Combining these results together, we have 
\eqref{aeq:err1} and \eqref{eq:err2}.
\end{proof}

Recall Mertens' second theorem in the following forms. 

\begin{proposition}\label{mertensexplicit}
Let $x>1$. We have
\[\sum_{p \leq x} \frac{1}{p} = \log{\log{x}} + M + M'(x),\]
where $M \approx 0.2614\ldots$ and 
\[\lvert M'(x) \rvert \leq \frac{1}{\log^2(x)}.\]
\end{proposition}
\begin{proof}
This is the Corollary to Theorem 5 in \cite{rosser1962}.
\end{proof}
\begin{proposition}\label{mertensexplicit1}
Let $x\ge 2$. We have
\begin{equation}
\log{\log{x}} +0.2614\leq \sum_{p \leq x} \frac{1}{p} \leq  
\log{\log{x}} +0.8666.    
\end{equation}
\end{proposition}
\begin{proof}
The bounds follow from Theorem 5 in \cite{rosser1962} and some simple computations. Note that the upper bound is optimal, with equality occurring at $x = 2$.
\end{proof}
We also introduce a helpful estimate.
\begin{proposition}\label{lsquaresum}
Let $x>1$. We have
\[\sum_{p \leq x} \frac{\log^2{p}}{p} \leq (1+10^{-8})\frac{\log^2{x}}{2}.\]
\end{proposition}
\begin{proof}
For $1<x < 355991$, one may verify that
\[\sum_{p \leq x}\frac{\log^2{p}}{p} \leq \frac{\log^2{x}}{2}.\]
When $x \geq 355991$, we begin by applying partial summation to the sum in question
\begin{align}\label{psl2}\sum_{p \leq x} \frac{\log^2{p}}{p} &= \pi(x) \frac{\log^{2}{x}}{x} + \int_2^{355991}\pi(t)\left(\frac{\log^2{t}-2\log{t}}{t^2} \right)dt   \\ &+\int_{355991}^{x} \pi(t)\left(\frac{\log^2{t}-2\log{t}}{t^2}\right) dt .
\end{align}
One may compute the first integral exactly and find that it is bounded by 65.204. For the other instances of $\pi(t)$, it is suitable to use \cite{dusart1998}*{Theorem 1.10.7}, which states that, for $t \geq 355991$,
\begin{equation}\label{dusartpi}\pi(t) \leq \frac{t}{\log{t}} \left(1 + \frac{1}{\log{t}} + \frac{2.51}{\log^2{t}}\right).
\end{equation}
Taking \eqref{dusartpi} in \eqref{psl2} and simplifying, one arrives at 
\[\sum_{p\leq x} \frac{\log^2{p}}{p} \leq \frac{\log^2{x}}{2} \left(1 + \epsilon(x)\right),\]
where
\[\epsilon(x) \leq \frac{1.02\log\log{x}}{\log^2{x}} - \frac{8.808}{\log^2{x}} + \frac{15.06}{\log^3{x}}.\]
We observe that $\epsilon(x) < 0$ until $x > e^{e^{8.634}}$, and then $\epsilon(x)$ takes a maximum at $x_0 \approx e^{e^{9.134}}$. At this maximum, $\epsilon(x_0) \leq 10^{-8}$, establishing the result.
\end{proof}
We can now obtain an important explicit estimate.
\begin{lemma}\label{lemma:explicitbigconst}
Define $f(t) := \lvert \cos(t) - K \rvert$, where $K$ is defined in Theorem \ref{thm:forresttheorem}. Uniformly for $0 < \alpha \leq 1, \tau \in \R$, we have 
\begin{equation}
    \sum_{p \leq \exp(\sfrac{1}{\alpha})} \frac{f(\tau \log{p})}{p} \leq (1-K)\log{\frac{1}{\alpha}} + (2+2K)\log{\log(\lvert \tau \rvert + 3)} + C_0,
\end{equation}
where $C_0 = 7.28$.
\end{lemma}

\begin{proof}

We may assume $\tau > 0$. Start by considering $\tau \leq \alpha$. We observe that the Taylor expansion of $\cos{x}$ yields
\begin{equation}
  \lvert f(\tau \log{p}) - (1 - K) \rvert \leq \frac{1}{2}(\tau \log{p})^2.
\end{equation}
Hence, 
\begin{equation}\label{easytau}
    \sum_{p \leq w} \frac{f(\tau \log{p})}{p} \leq (1-K)\sum_{p \leq w} \frac{1}{p} + \frac{\tau^2}{2} \sum_{p \leq w}\frac{\log^2{p}}{p}.
\end{equation}
Applying Propositions \ref{mertensexplicit} and \ref{lsquaresum} to \eqref{easytau}, we obtain
\begin{equation}\label{fbaseestimate}
\sum_{p \leq \exp(\sfrac{1}{\alpha})} \frac{f(\tau \log{p})}{p} \leq (1-K)\left(\log{\frac{1}{\alpha}} + M + 
M'(\exp(\sfrac{1}{\alpha}))\right) + \frac{(1+10^{-8})}{4}\left(\frac{\tau}{\alpha}\right)^2 .
\end{equation}

Let $c > 1$ be a constant that will be chosen later. When $\tfrac{\tau}{c} \leq \alpha \leq 1$, we~have
\begin{equation}\label{eqn:firstcase}
    \sum_{p \leq \exp(\sfrac{1}{\alpha})} \frac{f(\tau \log{p})}{p} \leq (1-K)\log{\frac{1}{\alpha}} + (1-K)(M+1) + \frac{(1+10^{-8})c^2}{4}.
\end{equation}

Now, we consider $\alpha < \tfrac{\tau}{c} \leq 1$. If $w = \exp(\frac{c}{\tau})$, then \eqref{easytau} yields
\begin{equation}\label{1tow}
\begin{aligned}
\sum_{p \leq w} \frac{f(\tau \log{p})}{p} &\leq (1-K)\left(\log{\log{w}} + M + M'(w)\right) + \frac{(1+10^{-8})}{4} \\ &\leq (1-K)\log{\log{w}} + (1-K)(M+1) + \frac{(1+10^{-8})c^2}{4}.
\end{aligned}
\end{equation}
Noting that $\overline{f} = 1-K$, $S(f) = 1+K$, and $V(f) = 4$, we can now take $z = \exp(\alpha^{-1})$ in Lemma \ref{boundedvariationest}. This yields
\begin{equation}\label{wtoz}
    \sum_{w < p \leq z} \frac{f(\tau \log{p})}{p} \leq (1-K)\left(\log{\frac{1}{\alpha}} - \log{\log{w}}\right) + \lvert E_\tau(w) \rvert,
\end{equation}
where $E_\tau(w)$ is taken from \eqref{aeq:err1}.
Combining \eqref{1tow} and \eqref{wtoz} gives
\begin{equation}\label{eqn:secondcaseunrefined}
    \sum_{p \leq \exp(\sfrac{1}{\alpha})} \frac{f(\tau \log{p})}{p} \leq (1-K)\log{\frac{1}{\alpha}} + (1-K)(M+1)+\frac{(1+10^{-8})c^2}{4} + \lvert E_\tau(w) \rvert.
\end{equation}
For our choice of $c$, we focus on minimizing $\frac{(1+10^{-8})c^2}{4} + \lvert E_\tau(w) \rvert$. Taking $k_1 = 0$, we find that the best choice of $c$ is 2.67 and this leads to \eqref{eqn:secondcaseunrefined} becoming
\begin{equation}\label{eqn:secondcaserefined}
\sum_{p \leq \exp(\alpha^{-1})} \frac{f(\tau \log{p})}{p} \leq (1-K)\log{\frac{1}{\alpha}} + 7.28.
\end{equation}

If $\lvert \tau \rvert > 1$, we first consider the case that $(1+\epsilon)6.455\log^2(\lvert \tau \rvert + 3) \leq \tfrac{1}{\alpha}$. Taking $w$ as in \eqref{eq:wcomp} and $z = \exp(\alpha^{-1})$ in Lemma \ref{boundedvariationest}, we obtain
\begin{equation}\label{wtoz2}
    \sum_{w < p \leq z} \frac{f(\tau \log{p})}{p} \leq (1-K)\left(\log(\sfrac{1}{\alpha}) - \log{\log{w}}\right) + E_\tau(w),
\end{equation}
where $\lvert E_\tau(w) \rvert$ is bounded in \eqref{eq:err2} (and therefore depends on choices of $\epsilon$ and $k_2$).
It follows trivially from Proposition \ref{mertensexplicit} and $f(\tau \log{p}) \leq 1+K$ that 
\begin{equation}\label{trivialw}
\begin{aligned}
    \sum_{p \leq w} \frac{f(\tau \log{p})}{p} &\leq (1+K)\left(\log{\log{w}}+M + \frac{1}{ \left( (1+\epsilon)6.455\right)^2\log^4(\lvert \tau \rvert + 3)}\right).
\end{aligned}
\end{equation}
Taking \eqref{wtoz2} and \eqref{trivialw} together, with our choice for $w$, yields
\begin{equation}\label{unoptimized}
\begin{aligned}
    \sum_{p \leq \exp(\alpha^{-1})} \frac{f(\tau \log{p})}{p} &\leq (1-K)\log{\frac{1}{\alpha}} +4K\log\log(\lvert \tau \rvert + 3) + 2K\log((1+\epsilon)6.455)\\ &+ (1+K)\left(M + \frac{1}{ \left( (1+\epsilon)6.455\right)^2\log^4(\lvert \tau \rvert + 3)}\right) + E_\tau(w).
    \end{aligned}
\end{equation}
We need to optimize the last three terms of \eqref{unoptimized} with respect to $\epsilon$ and $k_2$. For fixed $\epsilon$ and $\tau$, it appears that $O_{k_2,\epsilon}(1)$ as defined in \eqref{eq:R32} is decreasing in $k_2$, but the savings are slight for large enough $k_2$. Therefore, in the interest of simpler computations, we choose $k_2 = 3\cdot10^5$. Some rough optimization over the terms involving $\epsilon$ in \eqref{unoptimized} shows that $\epsilon = 3.61$ gives a relatively small maximum over these terms as a function of $\tau$. Making this choice of $\epsilon$ and bounding the terms by their maximum in $\tau$, we determine that
\begin{equation}\label{optimized}
     \sum_{p \leq \exp(\alpha^{-1})} \frac{f(\tau \log{p})}{p} \leq (1-K)\log{\frac{1}{\alpha}} +4K\log\log(\lvert \tau \rvert + 3) + 3.25.
\end{equation}

The final case to consider is $\lvert \tau \rvert > 1$ and $ (1+\epsilon)6.455\log^2(\lvert \tau \rvert + 3) > \tfrac{1}{\alpha}$, but in this case the sum in question is bounded by the sum estimated in \eqref{trivialw}. Given our choice of $\epsilon$, this implies 
\begin{equation}\label{eqn:thirdcase}
    \sum_{p \leq \exp(\alpha^{-1})} \frac{f(\tau \log{p})}{p} \leq (2+2K)\log\log(\lvert \tau \rvert + 3) + 4.87.
\end{equation}
Taking \eqref{eqn:secondcaserefined} as the worst case between \eqref{eqn:firstcase}, \eqref{eqn:secondcaserefined}, \eqref{optimized} and \eqref{eqn:thirdcase}, completes the proof.

\end{proof}
Here is interesting to note that, as it will be clear from the following results, the constant $C_0$ is the main contributor to the size of the constant in Theorem \ref{thm:forresttheorem}, and thus of $\delta (c)$. Thus reducing $C_0$ would be a good starting point to improve Theorem \ref{Theo:exp}.
\begin{lemma}\label{lambdalemma}
For $\alpha \in \left[0,1\right]$, define $\lambda := \lambda(\alpha)$ to be the real number satisfying 
\[\sum_{p \leq \exp(\alpha^{-1})} \frac{1-g(p)}{p}  = \lambda \sum_{p\leq \exp(\alpha^{-1})} \frac{1}{p}.\]
Then,
\begin{equation*}
  \begin{aligned}
    \Re \sum_{p \leq \exp(\alpha^{-1})} \frac{g(p)}{p^{1+i\tau}} &\leq (1-K\lambda)\log(\frac{1}{\alpha}) + (2+2K)(\log{\log{\lvert \tau \rvert + 3 }}) \\ &+ C_0 + (K-K\lambda)(M + M'(\exp(\sfrac{1}{\alpha}))),
    \end{aligned}
\end{equation*}
for any $\tau \in \R$.
\end{lemma}

\begin{proof}
Consider the identity
\[\Re\left(\frac{g(p)}{p^{i\tau}}\right) = g(p)(\cos(\tau \log{p}) - K) + Kg(p) \leq \lvert \cos(\tau \log{p}) - K \rvert = f(\tau \log{p}) + Kg(p).\]
The definition of $\lambda$ implies that
\begin{equation}\label{eqn:lambdadef}
\sum_{p \leq \exp(\alpha^{-1})} \frac{g(p)}{p} = (1- \lambda)\left(\log{\frac{1}{\alpha}} + M + M'(\exp(\sfrac{1}{\alpha}))\right).
\end{equation}
Therefore, 
\begin{equation*}
    \Re\left(\sum_{p \leq \exp(\alpha^{-1})} \frac{g(p)}{p^{1+i\tau}}\right) \leq \sum_{p \leq \exp(\alpha^{-1})} \frac{f(\tau \log{p})}{p}  + K \sum_{p \leq \exp(\sfrac{1}{\alpha})} \frac{g(p)}{p}.
\end{equation*}
The result follows by applying Lemma \ref{lemma:explicitbigconst} and \eqref{eqn:lambdadef} to the terms above.
\end{proof}

Let $F(s)$ be the Dirichlet series corresponding to $g(n)$. We have the following estimate.
\begin{lemma}\label{Fexplicit}
For $\Re(s) > 1$, we have
\begin{equation}
    \left| F(s) \right| \leq \exp(\nu_2)\cdot\exp\left\{\Re\left(\sum_p \frac{g(p)}{p^{s}}\right)\right\},
\end{equation}
where $\nu_2 = \gamma - M \leq 0.316$.
\end{lemma}

\begin{proof}
Since $g(n)$ is completely multiplicative, we have that
\[F(s) = \prod_p \left( 1  - \frac{g(p)}{p^s} \right)^{-1}.\]
Therefore,
\[ \left| F(s) \right|  =  \left| \exp\left(- \sum_p \log\left(1- \frac{g(p)}{p^s}\right) \right) \right|.\]
Applying the Taylor expansion of $\log(1-x)$ to the inside of the above sum, we obtain
\begin{equation}\label{Fpieces}\left| F(s) \right|  =  \left| \exp\left(\sum_p \sum_{k=1}^\infty \frac{1}{k}\left(\frac{g(p)}{p^s}\right)^k \right)\right|
= \left| \exp\left(\sum_p \frac{g(p)}{p^s}\right)\right| \cdot \left|\exp\left(\sum_{k=2}^\infty \frac{1}{k} \sum_p \frac{g(p)^k}{p^{ks}}\right) \right|.\end{equation}
The sum over primes in the right-most term can be bounded above by the "prime" zeta function
\[P(s) := \sum_{p} \frac{1}{p^s},\]
which converges for $\Re(s) > 1$.
Therefore, we have 
\begin{equation}\label{Fconstant}\left| \sum_{k=2}^\infty \frac{1}{k} \sum_p \frac{g(p)^k}{p^{ks}} \right| \leq \left| \sum_{k=2}^\infty \frac{P(ks)}{k} \right| \leq \sum_{k=2}\frac{P(k)}{k} = \gamma - M.\end{equation}
The equality in \eqref{Fconstant} follows from the definition of $B$, since
\[\gamma - M = \sum_p -\log\left(1-\frac{1}{p}\right)-\frac{1}{p} = \sum_{p} \sum_{k=2}^\infty \frac{1}{kp^k} = \sum_{k=2}^\infty \frac{P(k)}{k}.  \]
Inserting \eqref{Fconstant} into \eqref{Fpieces} yields the desired result.
\end{proof}
The following result will be used in bounding the sum over the primes in Lemma \ref{Fexplicit}.
\begin{lemma}
\label{lemma:v_1}
Uniformly for $0 < \alpha \leq 1$, we have
\begin{equation*}
    \sum_{\exp(\alpha^{-1})}^{\infty}\frac{1}{p^{1+\alpha}} \leq 0.9235=:v_1.
\end{equation*}
\end{lemma}
\begin{proof}
By partial summation we have
\begin{equation*}
    \sum_{\exp(\alpha^{-1})}^{\infty} \frac{1}{p^{1+\alpha}} =-\pi(\exp(\alpha^{-1}))\exp(-(1+\alpha^{-1}))+(1+\alpha) \int_{\exp(\alpha^{-1})}^{\infty} \frac{\pi(x)}{x^{2+\alpha}} dx.
\end{equation*}
Using (3.6) from \cite{rosser1962}, we then obtain
\begin{equation*}
\sum_{\exp(\alpha^{-1})}^{\infty} \frac{1}{p^{1+\alpha}}\leq (1+\alpha) 1.2551\int_{\exp(\alpha^{-1})}^{\infty} \frac{1}{ x^{1+\alpha}\log x} dx \leq \frac{(1+\alpha)1.2551}{e},
\end{equation*}
the result now follows taking the maximum over $\alpha \in (0,1]$.
\end{proof}
Note that using a better explicit version of the PNT could improve the above result, as this improvement appears to be minor we decided, for the sake of simplicity, for the above version.
\subsection{Proof of Theorem \ref{thm:forresttheorem}}\label{sub:3}

Consider $F(1+\alpha + it)$, where $0 < \alpha \leq 1$ and $t \in \R$. By Lemma \ref{Fexplicit}, we have
\begin{equation}\label{Falpha}
  \left| F(1+\alpha+it) \right| \leq \exp(\nu_2)\cdot\exp\left\{\Re\left(\sum_p \frac{g(p)}{p^{1+\alpha+it}}\right)\right\}.
\end{equation}
We break the sum over primes in \eqref{Falpha} at $\exp(\alpha^{-1})$, yielding the bound
\begin{equation}\label{hallF}
\begin{aligned}
\left| \Re\left(\sum_p \frac{g(p))}{p^{1+\alpha+it}}\right) \right| &=& &\left| \Re\left(\sum_{p\leq \exp(\alpha^{-1})} \frac{g(p)}{p^{1+\alpha+it}}\right) + \Re\left(\sum_{p > \exp(\alpha^{-1})} \frac{g(p)}{p^{1+\alpha+it}}\right)\right| \\ &\leq& &\left| \sum_{p\leq \exp(\alpha^{-1})} \frac{g(p)}{p^{1+\alpha+it}} \right| + \left| \sum_{p > \exp(\alpha^{-1})} \frac{1}{p^{1+\alpha}}\right|.
\end{aligned}
\end{equation}
Simply ignoring $p^\alpha$ in the first sum on the right of \eqref{hallF} and applying Lemma \ref{lemma:v_1} to the second sum, we obtain
\begin{equation}\label{realsum}
    \left| \Re\left(\sum_p \frac{g(p)}{p^{1+\alpha+it}}\right) \right| \leq \left| \sum_{p\leq \exp(\alpha^{-1})} \frac{g(p)}{p^{1+it}} \right| + \nu_1.
\end{equation}
Now, we may apply Lemma \ref{lambdalemma} to the remaining sum in \eqref{realsum} and place this estimate in \eqref{Falpha} to establish
\begin{equation}\label{explicitF}
    \lvert F(1+\alpha+i\tau) \rvert \leq \exp(C_0 + \nu_1 + \nu_2 + (K-K\lambda)(M + 1)).
\end{equation}
Write $C:= C_0 + \nu_1 + \nu_2 + K(M + 1)$. Recalling Theorem \ref{thm:matteotheorem}, we see that \eqref{explicitF} implies
\begin{equation*}
    H^2(\alpha) \leq \exp(2C)\alpha^{2K\lambda - 2}\sum_{k \in \Z} \frac{\log^{2+2K}(\lvert k \rvert + 4)}{(k-\sfrac{1}{2})^2 + 1}.
\end{equation*}
The integer sum above is a computable constant. Calling its square root $\nu_3$, we have
\begin{equation}\label{eqn:Hest}
    H(\alpha) \leq \nu_3 \exp(C) \alpha^{K\lambda -  1},
\end{equation}
and note that $\nu_3 \leq 4.36$. 
Now, if $\Lambda := \Lambda(x)$ is defined by 
\begin{equation}\label{lambdadef}
    \sum_{p\leq x} \frac{1 - g(p)}{p}  = \Lambda\sum_{p \leq x} \frac{1}{p},
\end{equation}
then, for $\sfrac{1}{\log{x}} \leq \alpha \leq 1$, 
\begin{equation}
\begin{aligned}
    \sum_{p \leq \exp(\alpha^{-1})}  \frac{1 - g(p)}{p} &\geq \sum_{p \leq x}  \frac{1 - g(p)}{p} - \sum_{\exp(\alpha^{-1}) < p \leq x}\frac{2}{p} \\ &\geq (\Lambda - 2) \sum_{p \leq x} \frac{1}{p} + 2 \sum_{p \leq \exp(\alpha^{-1})} \frac{1}{p}.
\end{aligned}
\end{equation}
Recalling the definition of $\lambda$ in Lemma \ref{lambdalemma} and using Proposition \ref{mertensexplicit1} we easily obtain,
\[\alpha^\lambda \leq \alpha^2(\log{x})^{2-\Lambda}e^{3(0.867-0.261)}\leq \alpha^2(\log{x})^{2-\Lambda}e^{1.82},\]
which, when applied to \eqref{eqn:Hest} implies
\begin{equation}
    H(\alpha) \leq \nu_3 \exp(C)e^{1.82K} \alpha^{2K-1} (\log{x})^{(2-\Lambda)K}.
\end{equation}
Taking this estimate for $H(\alpha)$ in Theorem \ref{thm:matteotheorem}, we find that
\begin{equation}\label{eqn:bigGest}
\begin{aligned}
   & \lvert G(x) \rvert \leq \left(3.14  \nu_3 \exp(C)e^{1.82K}  + \littleo{1}\right) \frac{x(\log{x})^{(2-\Lambda)K}}{\log{x}} \int_{\sfrac{1}{\log{x}}}^1 \! \alpha^{2K-2} \, d\alpha + \bigO_{x_0}\left(\frac{x}{\sqrt{\log{x}}}\right) \\  
    &= \left(\frac{3.14  \nu_3 \exp(C) e^{1.82K} }{1-2K} + \littleo{1}\right) \frac{x(\log{x})^{(2-\Lambda)K}}{\log{x}}\left( \log{x}^{1-K\Lambda} - \log{x}^{2-\Lambda}\right) + \bigO_{x_0}\left(\frac{x}{\sqrt{\log{x}}}\right) \\
    &\leq (\textbf{a} + \littleo{1}) \frac{x}{\log{x}}(\log{x})^{1-2K+(2-\Lambda)K} + \bigO_{x_0}\left(\frac{x}{\sqrt{\log{x}}}\right) \\
    &= (\textbf{a} +\littleo{1}) x (\log{x})^{-\Lambda K} + \bigO_{x_0}\left(\frac{x}{\sqrt{\log{x}}}\right). 
\end{aligned}
\end{equation}
Here, $\textbf{a} \approx 5.5\cdot10^5$ is the collected constant term up to this point. It follows from \eqref{lambdadef} that 
\[\Lambda = \frac{\sum_{p \leq x} \frac{1 - g(p)}{p} - \Lambda M- \Lambda M'(x)}{\log\log{x}},\]
and that for $x < 2$, we can take $\Lambda = 0$. For $x \geq 2$, we have $\Lambda \leq 2$. Furthermore, one may verify that $\lvert M'(x) \rvert < 0.6051$ for $2 \leq x < 4$ and $\lvert M'(x) \rvert \leq \tfrac{1}{\ln^2{4}} < 0.6051$ for $x > 4$. Therefore, we may write
\begin{equation}\label{eqn:constantstogether}
\begin{aligned}
    \textbf{a} (\log{x})^{-\Lambda K} &= \textbf{a}\exp(K\Lambda M) \exp\left(K\Lambda M'(x)\right)  \exp\left(-K \sum_{p \leq x} \frac{1 - g(p)}{p}\right) \\ 
        &\leq \textbf{a}\exp\left(2KM + {1.21K}\right) \exp\left(-K \sum_{p \leq x} \frac{1 - g(p)}{p}\right)
        \\ &= 9.75 \cdot 10^{5} \exp\left(-K \sum_{p \leq x} \frac{1 - g(p)}{p}\right).
\end{aligned}
\end{equation}

Taking \eqref{eqn:constantstogether} in \eqref{eqn:bigGest} completes the proof. 

\section*{Acknowledgments}
We would like to thank our supervisor, Tim Trudgian, for suggesting this topic to us and for his insightful comments. We are also grateful to Leo Goldmakher for his comments on an earlier version of this article and for bringing \cite{Mangerel} to our attention.
\bibliography{bibliographic.bib} 

\begin{bibdiv}
\begin{biblist}

\bib{Bennett}{article}{
      author={Bennett, M.~A.},
      author={Martin, G.},
      author={O'Bryant, K.},
      author={Rechnitzer, A.},
       title={Explicit bounds for primes in arithmetic progressions},
        date={2018},
     journal={Illinois J. Math},
      volume={6},
       pages={427\ndash 532},
}

\bib{Bordignon}{article}{
      author={Bordignon, M.},
      author={Kerr, B.},
       title={An explicit {P}\'{o}lya-{V}inogradov inequality via partial
  gaussian sums},
        date={2020},
     journal={To appear in Trans. Amer. Math. Soc.},
}

\bib{Breusch}{article}{
      author={Breusch, R.},
       title={Zur {V}erallgemeinerung des {B}ertrandschen {P}ostulates, da\ss
  zwischen {$x$} und 2 {$x$} stets {P}rimzahlen liegen},
        date={1932},
        ISSN={0025-5874},
     journal={Math. Z.},
      volume={34},
      number={1},
       pages={505\ndash 526},
         url={https://doi.org/10.1007/BF01180606},
      review={\MR{1545270}},
}

\bib{buchstab}{article}{
      author={Buchstab, A.~A.},
       title={On those numbers in an arithmetic progression all primes factors
  of which are small in magnitude},
        date={1949},
     journal={Doklady Akad. Nauk. SSSR},
      volume={67},
       pages={5\ndash 8},
}

\bib{deBruijn}{article}{
      author={de~Bruijn, N.~G.},
       title={The asymptotic behaviour of a function occurring in the theory of
  primes},
        date={1951},
        ISSN={0019-5839},
     journal={J. Indian Math. Soc. (N.S.)},
      volume={15},
       pages={25\ndash 32},
      review={\MR{43838}},
}

\bib{dusart1998}{thesis}{
      author={Dusart, P.},
       title={Autour de la fonction qui compte le nombre de nombres premiers},
        type={Ph.D. Thesis},
        date={1998},
}

\bib{Forrest}{article}{
      author={Francis, F.~J.},
       title={An investigation into several explicit versions of {B}urgess'
  bound},
        date={2019},
     journal={arXiv:1910.13669},
}

\bib{frolenkov2013}{article}{
      author={Frolenkov, D.~A.},
      author={Soundararajan, K.},
       title={A generalization of the {P}\'{o}lya-{V}inogradov inequality},
        date={2013},
        ISSN={1382-4090},
     journal={Ramanujan J.},
      volume={31},
      number={3},
       pages={271\ndash 279},
      review={\MR{3081668}},
}

\bib{Fromm}{article}{
      author={Fromm, E.},
      author={Goldmakher, L.},
       title={Improving the {B}urgess bound via {P}\'{o}lya-{V}inogradov},
        date={2019},
        ISSN={0002-9939},
     journal={Proc. Amer. Math. Soc.},
      volume={147},
      number={2},
       pages={461\ndash 466},
         url={https://doi.org/10.1090/proc/14171},
      review={\MR{3894884}},
}

\bib{Granville2007.1}{incollection}{
      author={Granville, A.},
      author={Soundararajan, K.},
       title={Negative values of truncations to {$L(1,\chi)$}},
        date={2007},
   booktitle={Analytic number theory},
      series={Clay Math. Proc.},
      volume={7},
   publisher={Amer. Math. Soc., Providence, RI},
       pages={141\ndash 148},
      review={\MR{2362198}},
}

\bib{hall1991}{article}{
      author={Hall, R.~R.},
      author={Tenenbaum, G.},
       title={Effective mean value estimates for complex multiplicative
  functions},
        date={1991},
        ISSN={0305-0041},
     journal={Math. Proc. Cambridge Philos. Soc.},
      volume={110},
      number={2},
       pages={337\ndash 351},
      review={\MR{1113432}},
}

\bib{hall1988}{book}{
      author={Hall, R.R.},
      author={Tenenbaum, G.},
       title={Divisors},
      series={Cambridge Tracts in Mathematics, Vol. 90},
   publisher={Cambridge University Press},
     address={Cambridge, New York, New Rochelle, Melbourne, and Sydney},
        date={1988},
        ISBN={9780521340564},
}

\bib{Hildebrand}{article}{
      author={Hildebrand, A.},
       title={Quantitative mean value theorems for nonnegative multiplicative
  functions. {II}},
        date={1987},
        ISSN={0065-1036},
     journal={Acta Arith.},
      volume={48},
      number={3},
       pages={209\ndash 260},
         url={https://doi.org/10.4064/aa-48-3-209-260},
      review={\MR{921088}},
}

\bib{Hildebrand1988}{article}{
      author={Hildebrand, A.},
       title={On the constant in the {P}\'{o}lya-{V}inogradov inequality},
        date={1988},
        ISSN={0008-4395},
     journal={Canad. Math. Bull.},
      volume={31},
      number={3},
       pages={347\ndash 352},
         url={https://doi.org/10.4153/CMB-1988-050-1},
      review={\MR{956367}},
}

\bib{Mangerel}{article}{
      author={Mangerel, A.~P.},
       title={Short character sums and the p{o}'lya-vinogradov inequality},
        date={2019},
     journal={arXiv:1905.09238},
}

\bib{montgomery1978}{article}{
      author={Montgomery, H.~L.},
       title={A note on the mean value of multiplicative functions},
        date={1978},
     journal={Inst. Mittag Leffler},
      volume={report N. 17},
}

\bib{Mossinghoff}{article}{
      author={Mossinghoff, M.~J.},
      author={Trudgian, T.~S.},
       title={Nonnegative trigonometric polynomials and a zero-free region for
  the {R}iemann zeta-function},
        date={2015},
        ISSN={0022-314X},
     journal={J. Number Theory},
      volume={157},
       pages={329\ndash 349},
         url={https://doi.org/10.1016/j.jnt.2015.05.010},
      review={\MR{3373245}},
}

\bib{Platt2019}{article}{
      author={Platt, D.},
      author={Trudgian, T.},
       title={The error term in the prime number theorem},
        date={2019},
     journal={To appear in Math. Comp.},
}

\bib{rosser1962}{article}{
      author={Rosser, J.~B.},
      author={Schoenfeld, L.},
       title={Approximate formulas for some functions of prime numbers},
        date={1962},
        ISSN={0019-2082},
     journal={Illinois J. Math.},
      volume={6},
       pages={64\ndash 94},
      review={\MR{0137689}},
}

\bib{tenenbaum2015}{book}{
      author={Tenenbaum, G.},
       title={Introduction to {A}nalytic and {P}robabilistic {N}umber
  {T}heory},
     edition={Third},
      series={Graduate Studies in Mathematics},
   publisher={American Mathematical Society, Providence, RI},
        date={2015},
      volume={163},
        ISBN={978-0-8218-9854-3},
        note={Translated from the 2008 French edition by Patrick D. F. Ion},
      review={\MR{3363366}},
}

\bib{trevino2015}{article}{
      author={Trevi\~{n}o, E.},
       title={The {B}urgess inequality and the least {$k$}th power
  non-residue},
        date={2015},
        ISSN={1793-0421},
     journal={Int. J. Number Theory},
      volume={11},
      number={5},
       pages={1653\ndash 1678},
      review={\MR{3376232}},
}

\bib{Trudgian}{article}{
      author={Trudgian, T.},
       title={Updating the error term in the prime number theorem},
        date={2016},
     journal={Ramanujan J.},
      volume={39},
       pages={225\ndash 234},
}

\bib{weil19481}{book}{
      author={Weil, A.},
       title={Sur {L}es {C}ourbes {A}lg\'{e}briques et {L}es {V}ari\'{e}t\'{e}s
  {Q}ui {S}'en {d}\'{e}duisent},
      series={Actualit\'{e}s Sci. Ind.},
   publisher={Hermann et Cie., Paris},
        date={1948},
         url={https://books.google.com.au/books?id=-Ip0nQEACAAJ},
}

\end{biblist}
\end{bibdiv}
\bibliographystyle{agsm}
\end{document}